\documentclass[10pt,a4paper,reqno]{amsart}

\usepackage{amsmath,amssymb,amsthm,amsfonts}
\usepackage{amsbsy}
\usepackage[utf8]{inputenc}

\usepackage[english]{babel}
\usepackage{url}
\usepackage{wrapfig}
\usepackage{enumitem}
\usepackage{multicol}
\usepackage{moreenum}
\usepackage{subcaption}

\usepackage{tikz}

\usepackage[a4paper,left=25mm,right=25mm,top=30mm,bottom=30mm]{geometry}

\usepackage{xcolor}

\definecolor{LinkColor}{rgb}{0,0,0} 
\usepackage[colorlinks=true,linkcolor=LinkColor,citecolor=LinkColor,urlcolor= LinkColor, naturalnames, hyperindex, pdfstartview=FitH, bookmarksnumbered, plainpages]{hyperref}

\newtheorem{theorem}{Theorem}[section]

\newtheorem{proposition}[theorem]{Proposition}

\theoremstyle{definition}

\newtheorem{remark}[theorem]{Remark}
\newtheorem{problem}[theorem]{Problem}

\theoremstyle{remark}
\newtheorem{example}[theorem]{Example}

\newcommand{\SL}{\operatorname{SL}}

\newcommand{\PSL}{\operatorname{PSL}}

\newcommand{\Aut}{\operatorname{Aut}}

\newcommand{\Tr}{\operatorname{Tr}}

\newcommand{\Z}{\textup{Z}}
\newcommand{\U}{\mathrm{U}}
\newcommand{\V}{\mathrm{V}}

\newcommand{\ZZ}{\mathbb{Z}}
\newcommand{\QQ}{\mathbb{Q}}

\title{From examples to methods: Two cases from the study of units in integral group rings}
\author{Andreas B\"achle}
\email{\href{mailto:ABaechle@gmx.net}{ABaechle@gmx.net}}
\author{Leo Margolis}
\address{Vakgroep Wiskunde, Vrije Universiteit Brussel, Pleinlaan 2, 1050 Brussels, Belgium}
\email{\href{mailto:leo.margolis@vub.be}{leo.margolis@vub.be}}
\thanks{The second author is a postdoctoral researcher of the Research Foundation Flanders (FWO - Vlaanderen).}

\begin{document}
\maketitle

\begin{abstract} In this article, we review the proofs of the first Zassenhaus Conjecture on conjugacy of torsion units in integral group rings for the alternating groups of degree $5$ and $6$, by Luthar-Passi and Hertweck. We describe how the study of these examples led to the development of two methods -- the HeLP method and the lattice method. We exhibit these methods and summarize some results which were achieved using them. We then apply these methods to the study of the first Zassenhaus conjecture for the alternating group of degree $7$ where only one critical case remains open for a full answer. Along the way we show in examples how recently obtained results can be combined with the methods presented and collect open problems some of which could be attacked using these methods.
\end{abstract}

\section{Introduction}
Ever since first being studied by G. Higman in his PhD thesis \cite{Hig40, San81} the structure of the unit group of an integral group ring $\ZZ G$ of a finite group $G$ has fascinated many researchers, as evident e.g.\ from several monographs devoted to the topic \cite{Seh78, Bov87, Seh93, PMS02, JdR16a, JdR16b}. A focal point of the study has been the understanding of units of finite order, i.e.\ \emph{torsion units}, and how their properties are determined by the group base $G$. Denote by $\U(RG)$ the unit group of the group ring $RG$ of the group $G$ over a ring $R$. Certainly, elements of the form $rg$, where $r$ is a unit in $R$ and $g$ an element of $G$, lie in $\U(RG)$ and these are known as \emph{trivial units}. The question on how far an element of $\U(\ZZ G)$ is from being trivial has been studied by various researchers. If $G$ is torsion-free, it is a long standing conjecture from the 1960's attributed to I. Kaplansky that any element of $\U(\ZZ G)$ is in fact a trivial unit \cite[Question 1.135]{Dni69}. On the other hand, when $G$ is a finite group, such a statement can not be expected, as it is easy to construct units of infinite order in $\U(\ZZ G)$ \cite[Chapter 1]{JdR16a} and using those, by conjugation, also non-trivial units of finite order. The strongest possible statement one could still expect was put forward in 1974 by H. Zassenhaus.\\

\noindent\textbf{(First) Zassenhaus Conjecture (ZC1):}\quad Let $G$ be a finite group. Then any unit of finite order in $\U(\ZZ G)$ is conjugate in the rational group algebra $\QQ G$ to a trivial unit of $\ZZ G$.\\

First investigations of (ZC1) focused on special classes of solvable, mostly metabelian, groups, cf.\ e.g.\ \cite{HP72, AH80, LB83, MRSW87}. The first result for a non-solvable group was achieved in 1989 \cite{LP89}.

\begin{theorem}[Luthar-Passi]\label{th:A5}
The Zassenhaus Conjecture holds for the alternating group of degree $5$.
\end{theorem}

While only proving (ZC1) for one group, the arguments in \cite{LP89} gave rise to a character-theoretic method of algorithmic nature which strongly influenced the study of $\ZZ G$ during the following decades. For instance, it played an important role in the proof of (ZC1) for cyclic-by-abelian groups \cite{Her08, CMdR13} and influenced the construction of the counterexamples to (ZC1) which finally solved the conjecture \cite{EM18}. After M. Hertweck gave a generalization of the method \cite{HertweckBrauer} it subsequently became known as the HeLP method (\textbf{He}rtweck-\textbf{L}uthar-\textbf{P}assi), a term coined by A. Konovalov. 

Several weaker variations of (ZC1) have been proposed, in particular regarding the orders of units in $\ZZ G$. As one is not interested in the orders coming from the coefficient ring $\ZZ$, these questions are usually formulated for normalized units. Define the augmentation homomorphism $\varepsilon: \ZZ G \rightarrow \ZZ$, $\sum_{g \in G} z_g g \mapsto \sum_{g \in G} z_g$ and the group of \emph{normalized units} in $\ZZ G$ as $\V(\ZZ G) = \{u \in \U(\ZZ G) \ | \ \varepsilon(u) =1 \}$. Then $\U(\ZZ G) = \pm \V(\Z G)$. It was already proved in the 1960's that if $p$ is a prime and $\V(\ZZ G)$ contains an element of order $p^n$ for some integer $n$, then $G$ also contains an element of order $p^n$ \cite{CL65}. The question whether the orders of torsion units in $\V(\ZZ G)$ coincides with the orders of elements in $G$ is often called the \emph{Spectrum Problem}. W. Kimmerle proposed to study what happens in the minimal situation when one is allowed to mix two primes \cite{Kim06}.\\

\noindent\textbf{Prime Graph Question (PQ):}\quad Let $G$ be a finite group. If $\V(\ZZ G)$ contains an element of order $pq$, does $G$ also contain an element of order $pq$, where $p$ and $q$ are primes? \\

When introduced, (PQ) was immediately proven to have a positive answer for solvable groups, a class of groups for which it seems impossible to decide whether (ZC1) holds for a given group in the class. But the main advantage for the study of (PQ) is that in contrast to other questions in the field, such as (ZC1) or the Spectrum Problem, a reduction theorem is available. Namely (PQ) holds for a group $G$ if it holds for all almost simple 
quotients
of $G$ \cite{KK16}. Here an \emph{almost simple} group denotes a group $A$ which can be sandwiched between a non-abelian simple group $S$ and the automorphism group of $S$, i.e. $S \cong \operatorname{Inn}(S) \leqslant A \leqslant \Aut(S)$. In this case $S$ is called the \emph{socle} of $A$. Hence one can hope that the Classification of Finite Simple Groups can provide additional insight into (PQ) or even allow a proof.  

It turned out that the HeLP method provides a strong tool to answer (PQ) for many groups (cf. Section~\ref{sec:HeLP}), but it was insufficient to prove it for the alternating group of degree $6$ \cite{Sal07}. Still, using a special additional argument it was proved in \cite{HertweckA6}:

\begin{theorem}[Hertweck]\label{th:A6}
The Zassenhaus Conjecture holds for the alternating group of degree $6$.
\end{theorem} 

It was realized by the authors that in fact Hertweck's argument could be generalized to provide a new method to study units in $\ZZ G$, the so-called lattice method \cite{BM17}. This method was then further developed and provided results for several new series' of groups. Thus each of the single groups studied by Luthar-Passi and Hertweck, respectively, gave rise to ideas which continue to deeply influence the study of units in group rings, emphasizing the importance of studying even single examples.

The aim of this article is to present the HeLP and lattice methods and some results achieved using them. In Sections~\ref{sec:HeLP} and \ref{sec:Lattice} we review the HeLP and the lattice method, respectively. Finally in Section~\ref{sec:A7} we give an example of the application of the methods to show how much can be achieved using them for torsion units in the integral group rings of the alternating group of degree 7. We also exhibit several problems which remain open in the subject, in particular some of which could be attacked using the methods described.

\section{HeLP}\label{sec:HeLP}
We will assume throughout that $G$ is a finite group. In general with regard to the order of units in $\ZZ G$ the following classical result, mentioned in the introduction, is available.

\begin{theorem}\label{th:CL} \cite{CL65}
The exponents of $G$ and $\V(\ZZ G)$ coincide.
\end{theorem}

Here the exponent of the infinite group $\V(\ZZ G)$ means the greatest common divisor of the orders of elements of finite order in $X$. Hence if we study questions such as (ZC1) and (PQ), we only have to consider units in $\V(\ZZ G)$ whose order divide the exponent of $G$. 

An important notion in the study of group rings are partial augmentations. Namely for a conjugacy class $C$  in $G$ and $u = \sum_{g \in G} z_g g$ an element in $R G$ set
$$\varepsilon_C(u) = \sum_{g \in C} z_g. $$
This is known as the \emph{partial augmentation} of $u$ at $C$. For (ZC1) we are interested whether a normalized torsion unit $u$ is \emph{rationally conjugate} to a group element $g \in G$, i.e.\ whether there is $x \in \U(\QQ G)$ such that $x^{-1}ux = g$. The significance of partial augmentations for this objective was realized in \cite{MRSW87} and \cite{LP89}.

\begin{theorem}\label{th:MRSW}\cite[Theorem 2.5]{MRSW87}
A unit $u \in \V(\ZZ G)$ of order $n$ is conjugate in $\QQ G$ to a trivial unit if and only if $\varepsilon_C(u^d) \geqslant 0$ for all conjugacy classes $C$ in $G$ and all divisors $d$ of $n$.
\end{theorem}

Hence it can be decided solely from the knowledge of the partial augmentations of $u^d$, $d$  running through the divisors of the order of a normalized torsion unit $u$, whether it is rationally conjugate to an element of $G$. A distribution $(\varepsilon_C(u^d))_{C,d}$ ($C$ running through the conjugacy classes of $G$ and $d$ through the divisors of $n$) of partial augmentations for an element $u$ of order $n$ is called \emph{trivial} if it coincides with the distribution of partial augmentations of an element of $G$. To keep the notation short, one frequently omits those partial augmentations $\varepsilon_C(u^d)$ that are zero or at least those that are a priori known to be zero, cf.\ Proposition~\ref{prop:partaugs}.

Now let $D: G \rightarrow \operatorname{GL}_n(K)$ be a representation of $G$ over a field $K$ of characteristic $0$ with character $\chi$. Then $D$ can be linearly extended to a ring homomorphism between $\ZZ G$ and the $n\times n$-matrix ring over $K$ and then restricted to $\U(\ZZ G)$ providing a group homomorphism $\U(\ZZ G) \rightarrow \operatorname{GL}_n(K)$. Then also $\chi$ extends to a character of $\U(\ZZ G)$ and for $u = \sum_{g \in G} z_g g \in \U(\ZZ G)$ we compute, using the fact that characters are class functions,
\begin{align}\label{eq:character}
\chi(u) = \sum_{g\in G} z_g\chi(g) = \sum_C \varepsilon_C(u) \chi(C) 
\end{align}
where the last sum runs over all conjugacy classes $C$ of $G$. Denote by $\chi|_U$ the restriction of $\chi$ to a subgroup $U$ of $\U(\ZZ G)$ and let $\xi$ be a character of $\langle u \rangle$ for $u \in \U(\ZZ G)$ of order $n$. Applying the scalar product of characters we obtain that
\begin{align}\label{eq:scalarprod}
\langle \chi|_{\langle u \rangle}, \xi \rangle =  \frac{1}{n} \sum_{i=1}^n \chi(u^i)\overline{\xi(u^i)} = \frac{1}{n} \sum_{i=1}^n \left(\sum_C \varepsilon_C(u^i)\chi(C) \right) \overline{\xi(u^i)} 
\end{align}
is a non-negative integer. Clearly, if $\xi$ has degree $1$, the number $\langle \chi|_{\langle u \rangle}, \xi \rangle$ is just the multiplicity of $\xi(u)$ as an eigenvalue of $D(u)$ and this point of view is often used when applying the HeLP method. We will denote the multiplicity of an $n$-th root of unity $\zeta$ as an eigenvalue of $D(u)$ as $\mu(\zeta, u, \chi)$.

Denote by $ \Tr_{E/F}: E \rightarrow F$ the number theoretical trace of the field extension $E/F$, i.e.\ $\Tr_{E/F}(e) = \sum_{\sigma \in \Aut(E/F)} \sigma(e)$. Then, by discrete Fourier transformation and the fact that multiplicites of eigenvalues are non-negative integers, we obtain the main formula of \cite{LP89} from \eqref{eq:scalarprod}.

\begin{theorem}\label{th:LP} \cite[Theorem 1]{LP89}
Let $u \in \V(\ZZ G)$ be of order $n$ and let $\zeta$ be a primitive complex $n$-th root of unity. Then for every complex character $\chi$ of $G$ and every integer $\ell$ the expression
$$\frac{1}{n} \sum_{d \mid n} \Tr_{\QQ(\zeta^d)/\QQ}\left(\chi(u^d)\zeta^{-d\ell} \right) = \frac{1}{n} \sum_{d \mid n} \Tr_{\QQ(\zeta^d)/\QQ}\left(\left(\sum_C \varepsilon_C(u^d) \chi(C) \right) \zeta^{-d\ell} \right) $$
is a non-negative integer.
\end{theorem} 
We hence obtain a system of linear inequalities with unknowns $\varepsilon_C(u^d)$, where $C$ runs through the conjugacy classes of $G$ and $d$ through the divisors of $n$. If one can show that $\varepsilon_C(u^d) \geqslant 0$ for all $C$ and all $d$, then, by Theorem~\ref{th:MRSW} we know that $u$ is conjugate in $\QQ G$ to a trivial unit. This shows that knowledge on the partial augmentations of units of finite order in $\ZZ G$ is fundamental for the application of the method. We summarize some a priori knowledge on partial augmentations.

\begin{proposition}\label{prop:partaugs} Let $u \in \V(\ZZ G)$ be of order $n$.
\begin{itemize}
\item[(i)] (Berman-Higman theorem) \cite[Proposition 1.5.1]{JdR16a} If $C= \{z\}$ is a central conjugacy class in $G$ and $u\neq z$, then $\varepsilon_C(u) = 0$. In particular $\varepsilon_{\{1\}}(u) = 0$, if $u \neq 1$.
\item[(ii)] \cite[Theorem 2.3]{HertweckBrauer} If $\varepsilon_C(u) \neq 0$, then the order of elements in $C$ divides $n$.
\item[(iii)] Let $p$ be a prime, $j$ a non-negative integer and $D$ a conjugacy class in $G$. Then
$$\sum_{C, \ C^{p^j} \subseteq D} \varepsilon_C(u) \equiv \varepsilon_D(u^{p^j}) \mod p.$$
Here $C^{p^j}$ means the $p^j$-th powers of elements in $C$ and the sum runs over the conjugacy classes $C$ of $G$.
\end{itemize}
\end{proposition} 

The last part of Proposition~\ref{prop:partaugs} is regarded as ``well-known'', but was often overlooked. See \cite[Proposition 3.1]{HeLPPaper} for a proof.

Though only a weaker version of Proposition~\ref{prop:partaugs} was available to I. S. Luthar and I. B. S. Passi, they were able to use this together with Theorem~\ref{th:LP} to prove (ZC1) for $A_5$, the alternating group of degree $5$, by going through the divisors of the exponent of $A_5$ and showing that the distributions of partial augmentations for normalized torsion units of the corresponding order are trivial, or that there is no normalized torsion unit at all, in case there is no element of that order in $A_5$. 
From the formulation of Theorem~\ref{th:LP} it is clear that the argument of Luthar and Passi can be applied for any group $G$ to obtain restrictions on the partial augmentations of units of finite order in $\V(\ZZ G)$. If the restrictions are strong enough, one obtains a solution of (ZC1) for $G$, but otherwise one still has knowledge on potential candidates to counterexamples to (ZC1) which can be further investigated using other arguments. One such special argument to prove (ZC1) for the symmetric group $S_5$ was used in \cite{LT91}. 

If now one has knowledge on generic characters of a series of groups, then one can try to apply Theorem~\ref{th:LP} simultaneously to all these groups. The first to follow this approach was R. Wagner in his Master thesis \cite{Wag95} (cf.\ also \cite[Section 6]{BHK04}). He showed for $G = \PSL(2,p^f)$, where $f \leqslant 2$ and $\PSL(n,q)$ denotes the projective special linear group of dimension $n$ over a field with $q$ elements, that if the order of $u \in \V(\ZZ G)$ is divisible by $p$, then $u$ is in fact of order $p$ and conjugate in $\QQ G$ to a trivial unit. Relying on Theorem~\ref{th:LP}, C. H{\"o}fert then showed that (ZC1) holds for all groups of order smaller than 71 \cite{BHK04}, using a special argument for a central extension of $S_4$ of order $48$.

For a long time $A_5$ remained the only simple group for which (ZC1) was known. This only changed when Hertweck proved (ZC1) for $\PSL(2,7)$ \cite[Example 3.6]{Her06}, again relying on Theorem~\ref{th:LP}, but now also on the full strength of Proposition~\ref{prop:partaugs}, and -- again -- a special argument. He showed however that in this case the special argument was in fact reflecting a more general approach to the arguments of Luthar and Passi. Namely, he showed in \cite{HertweckBrauer} that it is possible to also use $p$-Brauer characters $\chi$ in \eqref{eq:character}, \eqref{eq:scalarprod} and Theorem~\ref{th:LP} if they are applied for a unit $u$ of order not divisible by $p$ and the sums are understood to run only on conjugacy classes containing $p$-regular elements, i.e.\ elements of order not divisible by $p$. From this extension the name HeLP originated.

Note that Hertweck's generalization is only useful, compared to the classical version, if there is a $p$-Brauer character which does not coincide with an ordinary character on all $p$-regular classes. Such $p$-Brauer characters do not exist e.g.\ if $G$ is $p$-solvable \cite[Theorem 22.1]{CR81}. But in particular for simple groups it provides a huge improvement and it was immediately used in \cite{HertweckBrauer} to (re)prove (ZC1) for $S_5$, $\PSL(2,7)$, $\PSL(2,11)$ and $\PSL(2,13)$ and to prove (PQ) for $\PSL(2,p)$ for any prime $p$. The potential of the method to study (PQ) algorithmically was then realized by V. Bovdi and A. Konovalov who used it to prove (PQ) for 13 sporadic simple groups in a series of papers, cf.\ e.g.\ \cite{BKL08}. It was also used by other authors to prove (ZC1) or (PQ) for single groups, e.g.\ in \cite{Gil13} or \cite{KK16}. Other successful applications were given in \cite{4primaryI} proving (PQ) for almost simple group with socle isomorphic to $\PSL(2,p^f)$ for $f \leqslant 2$ or to prove (ZC1) for $\PSL(2,p)$ if $p$ is a Fermat or Mersenne prime \cite{MdRS19} and for $\SL(2,p^f)$ for any prime $p$ and $f \leqslant 2$ \cite{dRS19}, the first infinite series of almost simple groups for which (ZC1) was proved.

The method was also frequently applied to the study of (PQ), sometimes supplemented by the lattice method, see Section~\ref{sec:Lattice}. We will first comment on some more aspects and variations of the method. 

\subsection{Implementations}
While first applications of the HeLP method used manual calculations, the closed formula of Theorem~\ref{th:LP} demonstrates that it is also very accessible to computer calculations using a solver of integral linear inequalities. This is in particular unavoidable if one wants to study many groups at the same time, e.g.\ all groups of small order, or groups where the character values are rather big and not feasible for manual calculations, such as sporadic simple groups.

This led several researchers to implement the HeLP method as a program which takes a group $G$ and its character tables as input and produces the restrictions on the partial augmentations of units of finite order in $\V(\ZZ G)$ which one gets from Theorem~\ref{th:LP}. To our knowledge such programs were written by C. H\"ofert \cite{BHK04}, V. Bovdi and A. Konovalov, cf.\ e.g.\ \cite{BKL08}, A.~Herman and G.~Singh \cite{HS15}.  Finally, the authors \cite{HeLPPaper} developed a \textsf{GAP 4} package \cite{GAP} devoted to the method\footnote{This package is free and open software, comes with recent \textsf{GAP} installations and is also available from \url{https://gap-packages.github.io/HeLP/}.  Note that some of the external software might need compilation, see the manual.}. It can be used in particular by researchers not in possession of their own implementation to verify results achieved by the method. In Sections~\ref{sec:Lattice} and \ref{sec:A7} we will give examples on how to use this package.

It should be remarked that no implementation can be seen as ``final''. On one hand improvement on the theoretical knowledge about partial augmentations of units of finite order, such as collected in Proposition~\ref{prop:partaugs}, can lead to new restrictions which have to be taken into account. On the other hand, when one is interested in applying the method to groups with ``big'' character values, e.g. the Monster group, one needs very powerful solvers of integral inequalities. The package written by the authors uses the solvers \textsf{4ti2} \cite{4ti2} and \textsf{normaliz} \cite{Normaliz}. These are efficient solvers, but other solvers might be better adjusted for special systems of linear inequalities.

\subsection{A non-cyclic variation}
Of course, it is not only of interest to study units of finite order in $\V(\ZZ G)$, or in other words finite cyclic subgroups of $\V(\ZZ G)$, but also arbitrary finite subgroups of $\V(\ZZ G)$. The second and third Zassenhaus Conjecture, which were intensively studied, made predictions about conjugation of such subgroups in the corresponding rational group algebra. Various special types of such subgroups were considered in the last decades, such as group bases, i.e.\ finite subgroups $H$ of $\V(\ZZ G)$ such that $\ZZ G = \ZZ H$, $p$-subgroups of $\V(\ZZ G)$ or composition factors of finite subgroups in $\V(\ZZ G)$. We will not go into detail on all these topics, but only mention where variations of the HeLP method played a role in their study. We refer the interested reader to \cite{MdRSurvey} for a survey touching upon those topics, to Hertweck's solution of the Isomorphism Problem for background on the case of group bases \cite{Her01} and to \cite{KM17} for a survey on $p$-subgroups. 

Looking at equation \eqref{eq:scalarprod} one directly sees that in fact $\langle u \rangle$ can be replaced in this equation by any finite subgroup $U$ of $\V(\ZZ G)$ on the left-hand side, adjusting the right-hand side accordingly, i.e.\ summing over all the elements of $U$ instead of the distinct powers of $u$. Hence one again obtains linear inequalities involving the partial augmentations of elements in $U$ as unknowns. It is not hard to see that also Hertweck's variation using Brauer characters can be generalized to non-cyclic groups $U$ \cite[Lemma 2.2]{Mar17}.

This idea was applied for the first time in the PhD thesis of H\"ofert \cite{HoefertDoktor} to study subgroups and composition factors of subgroups in $\V(\ZZ G)$ mainly for $G = \PSL(2,q)$, see \cite{HHK09} for the main results. It was shown in particular that $2$-subgroups of $\V(\ZZ \PSL(2,q))$ are isomorphic to subgroups of $\PSL(2,q)$, composition factors of finite subgroups of $\V(\ZZ \PSL(2,q))$ are isomorphic to subgroups of $\PSL(2,q)$ and subgroups of $\V(\ZZ G)$ of order smaller than $G$ are solvable if $G$ is a minimal simple group. It was then also applied in \cite{BK11, BM15} to obtain information on $p$-subgroups in $\V(\ZZ G)$ where $G$ is a minimal simple group or when $G = \PSL(2,q)$ and $p$ is the defining characteristic of $\PSL(2,q)$.

As mentioned above, if $\V(\ZZ G)$ contains a subgroup isomorphic to $C_{p^n}$, for some prime $p$, then also $G$ contains a subgroup isomorphic to $C_{p^n}$ and this holds for any group $G$ \cite{CL65}. The question for which other groups this happens was posed in \cite[Problem~19]{JMNK07}. \\

\noindent\textbf{Subgroup Isomorphism Problem (SIP):}\quad For which groups $U$ does the following statement hold for any finite group $G$? If $U$ is isomorphic to a subgroup of $\V(\ZZ G)$, then $U$ is isomorphic to a subgroup of $G$.\\

Only very little is known regarding (SIP). If $\ZZ G \cong \ZZ H$ for non-isomorphic groups $G$ and $H$, then (SIP) has a negative answer for $G$ and $H$. Hence Hertweck's counterexamples to the isomorphism problem \cite{Her01} are examples of groups where (SIP) has a negative answer. They remain the minimal groups known for which (SIP) has a negative answer. On the other hand, a positive answer has only been obtained for $C_{p^n}$, $C_p \times C_p$, for $p$ a prime, and $C_4 \times C_2$. The group $C_2 \times C_2$ was handled by Kimmerle \cite{JMNK07} and is based on the Brauer-Suzuki theorem, hence relying on a deep group-theoretical result. In contrast to this, the case of $C_p \times C_p$ for odd primes $p$ is a direct application of the non-cyclic HeLP method \cite{HertweckCpCp}. The case of $C_4 \times C_2$ uses classical group-theoretical results to reduce the question to certain series' of groups which are then studied using the non-cyclic HeLP method \cite{Mar17}. It seems plausible that much more could be said using the method, in particular if one weakens (SIP) to only include solvable groups $G$. The following case is of interest for (SIP).

\begin{problem} Does (SIP) hold for $C_p \times C_p \times C_p$? Does it hold for $S_3$, $Q_8$, $D_8$ or $A_5$? \end{problem}

One can not hope that the non-cyclic HeLP method alone will solve this problem, as demonstrated in \cite[Proposition 4.3]{KM17}. But the following variations might be more approachable.\\

\noindent\textbf{Problem~\thetheorem.1.}\ Does (SIP) hold for $C_p \times C_p \times C_p$, if one assumes that the group base $G$ is solvable?\\

\noindent\textbf{Problem~\thetheorem.2.}\  Does (SIP) hold for $C_p \times C_p \times C_p$, if one assumes that (ZC1) is valid for $\mathbb{Z}G$ where $G$ denotes the group base?\\

Note that (PQ) is in fact the specialization of (SIP) to the group $U = C_{pq}$.

The potential of the non-cyclic HeLP method is far from having been fully exhausted. In particular composition factors of finite subgroups of $\V(\ZZ G)$ have not been studied in-depth.

\subsection{Extending coefficients}\label{subsect:extcoeff}

There is clearly no need to restrict our interest to coefficients coming from the rational integers $\mathbb{Z}$, when looking at torsion units in integral group rings. Sometimes it is more convenient or interesting to allow coefficients from a ring of algebraic integers $\mathcal{O}$ in a number field $K$. In this case there is an obvious generalization of the first Zassenhaus Conjecture:\\

 \noindent\textbf{(First) Zassenhaus Conjecture for rings of algebraic integers:}\quad  Let $G$ be a finite group and let $\mathcal{O}$ be the ring of algebraic integers in a number field $K$. Then any unit of finite order in $\U(\mathcal{O} G)$ is conjugate by a unit of $KG$ to a trivial unit of $\mathcal{O} G$. \\
 
 From now on let $\mathcal{O}$ be the ring of algebraic integers in a fixed number field $K$.  As usual we denote by $\V(\mathcal{O}G)$ the group of units of augmentation $1$ in $\mathcal{O}G$. As it is the case for normalized torsion units in $\ZZ G$, the order of any torsion element in $\V(\mathcal{O}G)$ divides the exponent of $G$. Also the vanishing results on torsion units as in Proposition~\ref{prop:partaugs} apply for elements of $\mathcal{O}G$ and the criterion in Theorem~\ref{th:MRSW} for conjugacy of torsion units in $\V(\mathcal{O}G)$ to elements of $G$ by units of $KG$ using distributions of partial augmentations holds with the obvious adaptions; see \cite[5.4]{BKS20} for references.

Now one can prove that also the values of partial augmentations of elements of $\mathcal{O}G$ are governed by $G$ (and not by $\mathcal{O}$). Namely by  \cite[Lemma~5.5]{BKS20}, for a torsion units $u \in \V(\mathcal{O}G)$ we have $\varepsilon_C(u) \in \mathcal{O} \cap \ZZ[\zeta]$, where $\zeta$ is a primitive $n$-th root of unity if the order of $u$ is $n$ (which then is a divisor of the exponent of $G$ by the above). 

The proofs of the formulas for the muliplicities of the eigenvalues of the image of a torsion unit $u \in \V(\mathcal{O}G)$ under a representation $D$ of $G$ as in Theorem~\ref{th:LP} carry over verbatim (in case $D$ is a modular representation, we require that the characteristic does not divide the order of $u$).

Assume that a potential torsion unit $u \in \V(\mathcal{O}G)$ of order $n$ is given and $\zeta$ is a primitive $n$-th complex root of unity. Choose a $\mathbb{Z}$-basis $B$ of $\mathcal{O} \cap \ZZ[\zeta]$. Then one can express each $\varepsilon_x(u) = \sum_{b \in B}\alpha_{x, b}b$ with $\alpha_{x, b} \in \ZZ$. If we assume that the character values of the proper powers of $u$ under $\chi$ are known, then we are interested in the values of $\alpha_{x,b}$.  Then one can again compute the multiplicity of powers of $\zeta$ as an eigenvalue of the image of $u$ under a representation with character $\chi$ and one gets a system of linear inequalities for the $\alpha_{x,b}$ over $\ZZ$:
\begin{equation*} 
\sum_{x^G}\sum_{b \in B}\alpha_{x, b}\Tr_{\QQ(\zeta)/\QQ}\left(\chi(x)\zeta^{-\ell}b\right) +  \sum_{1 \not= d \mid n} \Tr_{\QQ(\zeta^d)/\QQ}\left(\chi(u^d)\zeta^{-d\ell}\right) \in n\mathbb{Z}_{\geqslant 0}.  
\end{equation*}
Compared to the ``classical'' HeLP method, where $\mathcal{O} = \ZZ$, the number of variables grows by a factor which equals the degree of $K\cap\QQ(\zeta)$ over $\QQ$. This system of inequalities can now also be used to obtain restrictions on partial augmentations on torsion units in $\mathcal{O}G$ for a fixed group $G$. Since the possible orders of torsion elements in $\V(\mathcal{O}G)$ divide the exponent of $G$ and all partial augmentations of these units are contained in $\mathbb{Z}[\zeta]$, where $\zeta$ is a primitive root of unity of order the exponent of $G$, there are actually only finitely many cases to handle to prove the (first) Zassenhaus Conjecture variant for \emph{all} rings of algebraic integers. This was implemented in \textsf{GAP}. With this implementation it was for instance proven in \cite{BKS20} that this variation holds for all groups of order at most $95$ except if the group is isomorphic to $A_5$, the natural wreath product $S_3 \wr S_2$ or if it maps onto $S_4$.

The first Zassenhaus Conjecture variation for rings of algebraic integers has implications on (ZC1) through the following result \cite[Proposition~8.2]{Her08}:
        
\begin{proposition}[Hertweck]\label{prop:Hertweck_reduction} Let $G$ be a finite
  group, let $A$ be a finite abelian group of exponent $m$ and $\zeta_m$ a primitive complex $m$-th root of unity.
Suppose that any torsion element of $\U(\ZZ[\zeta_m]G)$ is conjugate in $\U(\QQ(\zeta_m)G)$ to a trivial unit of $\ZZ[\zeta_m]G$. Then (ZC1) holds for $G \times A$. \end{proposition}

 This implies together with the above mentioned results that (ZC1) holds for $G \times A$, where $A$ is any finite abelian group, if $|G| \leqslant 95$ and $G$ does not map onto $S_4$ and is different from $A_5$ and $S_3 \wr S_2$, see \cite[Proposition~5.8]{BKS20}. This immediately raises the following problem:

 \begin{problem} Solve the first Zassenhaus Conjecture for rings of algebraic integers for the groups $S_4$, $A_5$ and $S_3 \wr S_2$.\end{problem}
 
The few remaining distributions that need to be ruled out can be found in \cite[Section~5]{BKS20}.
 
 In general the behavior of (ZC1) under direct products is rather poorly understood. One of the few general results that (ZC1) remains valid under direct products is the following \cite[Proposition~8.1]{Her08}. (Recall that (ZC1) holds for nilpotent groups \cite{Wei91}.)
 
\begin{proposition}[Hertweck] Let $G$ be a finite group for which (ZC1) holds, and let $\Pi$ be a nilpotent group of order coprime to the order of $G$. Then (ZC1) holds for the direct product $G \times \Pi$.
 \end{proposition}

\subsection{Further possible applications}

Although the HeLP method has been applied in different ways by many researchers, it seems very plausible it still can provide new insight into different questions for subgroups of $\V(\ZZ G)$, in particular when one uses new knowledge available on partial augmentations and studies questions which have not been in the focus of attention so far.

For example, the following result regarding partial augmentations of units of finite order is far from having been exploited in its full strength.

\begin{theorem}[Hertweck]\label{th:padic}  Let $u \in \V(\ZZ G)$ be a unit of finite order and $p$ a prime.
\begin{enumerate}
\item Assume that the $p$-part of $u$ is conjugate in $\ZZ_p G$, the $p$-adic group ring of $G$, to an element $h \in G$. If $\varepsilon_C(u) \neq 0$, then the $p$-part of an element in $C$ is conjugate in $G$ to $h$ \cite[Lemma 2.2]{Her08}.
\item Assume $N$ is a normal $p$-subgroup of $G$ such that $u$ maps to the identity under the natural projection $\ZZ G \rightarrow \ZZ (G/N)$. Then $u$ is conjugate in $\ZZ_p G$ to an element of $G$ \cite{MargolisHertweck}.
\end{enumerate}
\end{theorem} 

We provide an example of how this result can be applied:

\begin{example}
 Let $G$ be the semidirect product of the faithful irreducible $\mathbb{F}_5S_3$-module with $S_3$, $ G = (\mathbb{F}_5 \times \mathbb{F}_5) \rtimes S_3$. This is a group of order $150$ with \textsf{GAP}-Id \texttt{[150,5]}. By \cite{BHKMS18} the existence of a unit of order $10$ in $\V(\ZZ G)$ which has non-vanishing partial augmentations at two different conjugacy classes of elements of order $10$ can not be excluded using HeLP and other standard methods. Assume such a unit $u$ exists. Now $G$ has a normal Sylow $5$-subgroup $N$ and it follows that the $5$-part of $u$ maps to the identity under the map $\ZZ G \rightarrow \ZZ (G/N)$. So by the previous theorem the $5$-part of $u$ is conjugate to an element of $G$ in $\ZZ_5 G$. Hence, by the first part of the theorem, if $\varepsilon_C(u) \neq 0$ and $\varepsilon_D(u) \neq 0$, then the $5$-parts of elements in $C$ and $D$ are conjugate in $G$. But this is not the case for any two conjugacy classes of elements of order $10$ in $G$, contradicting the existence of $u$. Since this was the only case that was left open for (ZC1) in \cite{BHKMS18} we conclude that in fact (ZC1) holds for $G$.
\end{example}

Using the character-theory of supersolvable groups one could try to see how much the HeLP method, combined also with Theorem~\ref{th:padic}, can do to solve the following.

\begin{problem} Is (ZC1) true for supersolvable groups?\end{problem} 

We note that the counterexamples to (ZC1) given in \cite{EM18} are not too far from being supersolvable at first sight: they are of the form $(C_{pq} \times C_{pq}) \rtimes A$ for certain primes $p$ and $q$ and an abelian group $A$. It is however of utter importance for the construction of the examples that they possess minimal normal elementary-abelian subgroups for two different primes which are both not cyclic.

The following open problem has not been much studied, as (ZC1) was the central conjecture in the area for many years. With the solution of (ZC1) it now deserves more attention in our opinion.\\

\noindent\textbf{Kimmerle Problem (KP):}\quad Let $G$ be a finite group. Is any unit of finite order in $\U(\ZZ G)$ conjugate in the rational group algebra $\QQ H$ to a trivial unit of $\ZZ G$ for $H$ a finite group containing $G$? \\

Since (KP) turned out to be equivalent to a problem initiated by A.~Bovdi, (KP) is sometimes also abbreviated as (Gen-BP).
The following result gives a concrete approach to (KP) \cite[Proposition 2.1 and its proof]{MdR19}:

\begin{proposition} Let $u \in \V(\ZZ G)$ be a unit of order $n$.
\begin{enumerate}
\item There exists a finite group $H$ containing $G$ such that $u$ is conjugate in $\QQ H$ to an element of $G$ if and only if the sum of coefficients of $u$ with respect to elements of $G$ of order $m$ is $0$ for all $m \neq n$.
\item If (KP) holds for $G$, then one can take $H$ to be the symmetric group on $G$ in which $G$ is embedded by its regular action.
\end{enumerate}
\end{proposition}

This clarifies how tools developed for the study of (ZC1) can be applied to answer (KP). For instance, an immediate consequence of \cite{BHKMS18} and the previous proposition is that (KP) has a positive answer for all groups of order smaller than $288$. (KP) has a positive answer for all currently known counterexamples to (ZC1) \cite[Remark~3.4]{BKS20}. Moreover, (KP) has an affirmative answer for all supersolvable groups \cite[Theorem~3.6]{BKS20}. This suggests to study the following.

\begin{problem} Does (KP) have a positive answer for solvable groups?\end{problem}

The result on groups of small order using extensions of coefficients in Section~\ref{subsect:extcoeff} was also powerful enough to prove that (KP) has an affirmative answer for $G \times A$, where $A$ is any finite abelian group and $G$ has order at most $95$. The exceptions that remaind for (ZC1) could be handled with a special argument \cite[Theorem~5.18]{BKS20}. In a similar manner, (KP) was verified for groups of the form $G\times A$, where $G$ is a Frobenius group and $A$ is any finite abelian group \cite[Theorem~5.17]{BKS20}. In this setting some additional arguments were needed to handle certain Frobenius complements. The first Zassenhaus Conjecture is still open for Frobenius groups in general.
\begin{problem}
 Does (ZC1) hold for Frobenius groups? Does it hold for $2$-Frobenius groups?
\end{problem}
A theorem of S. O. Juriaans and C. Polcino Milies \cite{JPM00} asserts that the order of a normalized torsion unit in the integral group ring of a Frobenius group $G$ either divides the order of the Frobenius kernel $K$ or the order of the Frobenius complement. With a classical approach, by reducing modulo a nilpotent normal subgroup, Hertweck proved that those with an order dividing $|K|$ are rationally conjugate to elements of $K$, see \cite{Her12}. For many Frobenius complements, the first Zassenhaus Conjecture is known, see the end of the introduction of the just mentioned article of Hertweck.

\section{Lattice}\label{sec:Lattice}

We will revisit Hertweck's proof of Theorem~\ref{th:A6} and explain how it can be understood to be a special instance of a more general method. We will include the commands from \textsf{GAP} and the \textsf{HeLP}-package which can be used to verify the calculations. For more details regarding installation and on how to explicitly use the commands we refer the reader to the extensive documentations of these programs.

When studying (ZC1) for $G = A_6$, the alternating group of degree $6$, one finds that the HeLP method is not sufficient for a proof: there remain problems with elements of order $6$. So, in fact, one can not even answer the Prime Graph Question. There is a potential unit $u \in \V(\ZZ G)$ of order $6$ such that $(\varepsilon_{\texttt{2a}}(u), \varepsilon_{\texttt{3a}}(u), \varepsilon_{\texttt{3b}}(u)) \in \{(-2,2,1), (-2,1,2) \}$ (using the \textsf{GAP}-notation for the conjugacy classes in \texttt{CharacterTable("A6")}) that cannot be excluded. This fact is manually computed in \cite{HertweckA6}. It can also easily be seen using the command \texttt{HeLP\_ZC(CharacterTable("A6"));} of the \textsf{HeLP}-package followed by \texttt{HeLP\_sol[6]}. The two sets of possible partial augmentations of units of order $6$ are mapped to each other under the outer automorphism of $G$, hence it suffices to exclude one of them. We will handle the case $(\varepsilon_{\texttt{2a}}(u), \varepsilon_{\texttt{3a}}(u), \varepsilon_{\texttt{3b}}(u)) = (-2,2,1)$ in which $u^2$ is conjugate in $\QQ G$ to an element of $\texttt{3b}$.

As explained in Section~\ref{sec:HeLP}, from the partial augmentations of $u$ and its powers we can compute the character values of $u$ under any ordinary representation of $G$. This is also implemented in the \textsf{HeLP}-package in the command \texttt{HeLP\_MultiplicitiesOfEigenvalues}. Denote by $\zeta$ a primitive complex 3-rd root of unity. $G$ has two non-equivalent irreducible $5$-dimensional complex representations, which in fact can be realized over $\ZZ$; they correspond to \texttt{Irr(C)[2]} and \texttt{Irr(C)[3]} where we set \texttt{C := CharacterTable("A6");}. Denote these representations by $D_{5a}$ and $D_{5b}$. Denote by $A \sim B$ the fact that two complex matrices are conjugate and by $\text{diag}(a_1,...,a_n)$ a diagonal matrix with diagonal entries $a_1$,...,$a_n$. Then
$$D_{5a}(u) \sim \text{diag}(1,\zeta,\zeta^2,-\zeta,-\zeta^2) \qquad \ \text{and} \qquad \ D_{5b}(u) \sim \text{diag}(1,\zeta,\zeta^2,-1,-1).$$

Let $R = \ZZ_{(3)}$ be the localization of $\ZZ$ at $3$, i.e.\ the rational numbers with denominator coprime to $3$, and $\mathbb{F}_3 = k = R/3R$ the residue field of $R$ modulo its maximal ideal. For the representations $D_{5a}$ and $D_{5b}$ we have corresponding simple $RG$-lattices $L_{5a}$ and $L_{5b}$ respectively which map to the $kG$-modules $\overline{L}_{5a}$ and $\overline{L}_{5b}$ modulo the ideal $3R$.  
From the decomposition matrix of $G$ modulo $3$, accessible via \texttt{DecompositionMatrix(C mod 3);}, we find that $\overline{L}_{5a}$ and $\overline{L}_{5b}$ are not simple any more, and their composition factors are the trivial module and a common $4$-dimensional simple $kG$-module $\overline{L}$. Here $\overline{L}$ corresponds to \texttt{Irr(C mod 3)[4]}. Using a special argument, Hertweck then shows that one can assume that $\overline{L}$ is in fact a submodule of both, $\overline{L}_{5a}$ and $\bar{L}_{5b}$. 

Clearly, $e = \frac{1+u^3}{2}$ is an idempotent in $RG$ and hence $L_{5b} = L_{5b}e \oplus L_{5b}(1-e)$. Note that $L_{5b}e$ is the part of $L_{5b}$ on which the involution $u^3$ acts trivially, while it acts non-trivially, namely as multiplication by $-1$, on $L_{5b}(1-e)$. So the action of $\langle u \rangle$ on $L_{5b}e$ corresponds to the eigenvalues $(1, \zeta, \zeta^2)$ and the action on $L_{5b}(1-e)$ to $(-1,-1)$. In particular, the $k\langle \overline{u} \rangle$-module $\overline{L}_{5b}(1-\overline{e})$ is the direct sum of two 1-dimensional modules, since $-1 \in k$. Furthermore, if we consider $\overline{L}$ as a submodule of $\overline{L}_{5b}$, then $\overline{L}_{5b}(1-\overline{e}) = \overline{L}(1-\overline{e})$. This can be seen since the other composition factor of $\overline{L}_{5b}$ is the trivial module, but $u^3$ has no fixed points on $\overline{L}_{5b}(1-\overline{e})$. Alternatively, this can also be concluded from the fact that the $3$-Brauer character corresponding to $\overline{L}$ is of degree $4$ and has value $0$ on \texttt{2a}, i.e.\ $\overline{L}\overline{e}$ and $\overline{L}(1-\overline{e})$ must both be $2$-dimensional. Overall the $k\langle u \rangle$-module $\overline{L}(1-\overline{e})$ is the direct sum of two $1$-dimensional modules.

Arguing analogously for $L_{5a}$ we find $\overline{L}_{5a}(1-\overline{e}) \cong \overline{L}(1-\overline{e})$ as $k\langle \overline{u} \rangle$-module. The eigenvalues corresponding to $L_{5a}(1-e)$ are $(-\zeta, -\zeta^2)$ which means that $L_{5a}(1-e)$ is an indecomposable $R\langle u \rangle$-module, as $R$ does not contain a primitive $3$-rd root of unity. By a classical theorem of Minkowski the map $\operatorname{GL}_2(R) \rightarrow \operatorname{GL}_2(k)$ does not contain an element of order $3$ in its kernel. This means that also $\overline{L}_{5a}(1-\overline{e})$ is indecomposable as $k\langle \overline{u} \rangle$-module, hence so is $\overline{L}(1-\overline{e})$. But in the last paragraph we found that $\overline{L}(1-\overline{e})$ is the direct sum of two $1$-dimensional modules, a contradiction with the existence of $u$. 

Let us recapitulate (in a different order): We started with a certain distribution of partial augmentations of a hypothetical unit $u \in \V(\ZZ G)$ of finite order divisible by $p = 3$. We calculated the eigenvalues of $u$ under two representations $D_1$, $D_2$ of $G$ in characteristic $0$ over a ring $R$ in which $p$ is not invertible and lies in the maximal ideal $M$. From these we derived knowledge on the structure of the $RG$-lattices $L_1$, $L_2$ corresponding to the representations when viewed as $R\langle u \rangle$-lattices. From these structures we derived information on the structure of $\overline{L}_1$ and $\overline{L}_2$ as $k\langle \overline{u} \rangle$-modules, where $k$ is the residue field $R/M$. We then looked at the composition factors of $\overline{L}_1$ and $\overline{L}_2$ as $kG$-modules and analyzed what our previous calculations imply for the structure of these composition factors when viewed as $k \langle \overline{u} \rangle$-modules. We found that for a certain common composition factor of $\overline{L}_1$ and $\overline{L}_2$ the possible $k \langle \overline{u} \rangle$-module structures were not compatible and this led to the final contradiction.

In a more general context this leads to the following method, refered to as \emph{lattice method}, introduced in \cite{BM17}:

\begin{enumerate}
\item Start with a group $G$ and a hypothetical unit $u \in \V(\ZZ G)$ of order $p^\ell m$ for which partial augmentations are known. Here $p$ is a prime, $\gcd(p, m) =1$ and $\ell \geqslant 1$. Moreover let $\chi$ be an irreducible ordinary character of $G$.

\item Let $K'$ be the smallest filed of characteristic $0$ which contains all character values of $\chi$ and a primitive root of unity of order $m'$ where $m'$ is the largest divisor of $|G|$ coprime to $p$. Let $R$ be the ring of integers of the $p$-adic completion of $K'$ and let $\zeta$ be an $m$-th primitive root of unity in $R$. The reason for the choice of this ring $R$ is on one hand a theorem of Fong which guarantees that there is an $R$-representation $D$ of $G$ with $RG$-module $M$ and character $\chi$ and at the same time that $p$ is as unramified as possible in a ring with this property. The motivation for the last fact will become clear later. One can also choose $R$ to be any complete discrete valuation ring of characteristic $0$ such that: $R$ contains a primitive $m$-th root of unity, $p$ is contained in the maximal ideal of $R$ and $p$ is unramified in the extension of the field of fractions of $R$ over $\mathbb{Q}_p(\chi)$. Here $\mathbb{Q}_p(\chi)$ denotes the extension of the $p$-adic rationals generated by the values of $\chi$. 
Let $k$ be the residue field of $R$. 

Note that in the example above we were working over the ring $\ZZ_{(3)}$ which was smaller, but it will follow that our choice is not restrictive as $\ZZ_{(3)}$ has all the needed properties.

\item Next we derive restrictions on the structure of $M$ as $R\langle u \rangle$-module from the eigenvalues of $D(u)$. If one uses only these eigenvalues, this is an impossible task in the generic case, but often it can be done in concrete situations. First of all note that from our choice of $R$ we can decompose $M$ as an $R\langle u \rangle$-module into a direct sum of $m$ modules $M_i$. These modules are the maximal submodules of $M$ on which $u^{p^\ell}$ acts as $\zeta^i$ for $1 \leqslant i \leqslant m$. I.e.\ the $M_i$ corresponds to the eigenvalues of $D(u)$ which have $p'$-part $\zeta^i$.  This makes the considerations on idempotents we had to do in the example above superfluous.

Now to get more restrictions on the isomorphism type of the $M_i$ one needs to study $RC_{p^\ell}$-modules. This is in general a wild problem, but when $\ell = 1$ and $p$ is unramified in $R$, then this can be done generically. Note that this is exactly the case in the example above. In general the bigger $\ell$ is and the more $p$ ramifies in $R$, the harder it is to describe the representation theory of $RC_{p^\ell}$. This explains our choice of $R$.

\item Using the bar-notation for the reduction modulo the maximal ideal of $R$, we then obtain restrictions on the structure of $\overline{M}$ as $k \langle \overline{u} \rangle$-module. Clearly $\overline{M}$ decomposes into a direct sum of the $\overline{M}_i$ and each of the $\overline{M}_i$ can be considered as a $kC_{p^\ell}$-module as the action of the $p'$-part on $\overline{M}_i$ is fixed. The representation theory of $kC_{p^\ell}$ is much easier to understand than that of $RC_{p^\ell}$: For each $j \leqslant p^\ell$ there is a unique indecomposable $kC_{p^\ell}$-module of dimension $j$ and this module is uniserial. One can think of $C_{p^\ell}$ acting on this module by a Jordan block with eigenvalue $1$ and size $j$. 

We now consider the composition factors of $\overline{M}$ as $kG$-module, but view them as $k \langle \overline{u}\rangle$-modules. Each of them decomposes also into a direct sum of $m$ modules which correspond to different $m$-th roots of unity in $k$ and we derive our knowledge on the possible isomorphism types of these modules from our knowledge on the structure of the $\overline{M}_i$. This last step can be made quit explicit using a combinatorial procedure involving Young tableaux. We will not describe it here, but just mention that it shows that if we do not use that the direct sum of certain $kG$-composition factors appears in the socle series of $\overline{M}$ as $kG$-module, then we can assume any composition series we like. In particular the argument from the example above that $\overline{L}$ was in fact a submodule of $\overline{L}_{5a}$ becomes obsolete.  
\end{enumerate}

We refer the reader to \cite{BM17, 4primaryII} for more technical details and references. A further example of the application of the method will be given in the next section. 
Although the lattice method was also used to prove (ZC1) in some instances where HeLP was not sufficient, apart from \cite{HertweckA6} this was done in \cite{BM17, 4primaryII}, it has so far proved to be of particular strength for the study of (PQ). One reason for this is that the representation theory of $RC_{p^\ell}$ is much more approachable when $\ell=1$, see (3) above.

The method was originally developed to complete the proof of (PQ) by Kimmerle and Konovalov for $3$-primary groups \cite{KK16}. Here a group is called $3$-primary, if its order is divisible by exactly three pairwise different primes. It was shown in \cite{KK16} that this result will follow if one can show that $\V(\ZZ G)$ does not contain elements of order $6$ for $G \in \{M_{10}, \operatorname{PGL}(2,9) \}$, which are both almost simple groups containing $A_6$ as a normal subgroup of index $2$. This was shown to be indeed true in \cite[Theorem 1.2]{BM17}.

In \cite{4primaryII} the combinatorics involved in step (4) of the method were developed and it was applied to answer (PQ) for many 4-primary groups, but also for (potentially) infinite series of groups (the infinity of the number of groups being a question of number theory). The analysis of a particular combinatorial situation coming up in the application of the HeLP method further led to the proof of (PQ) for three new sporadic simple groups for which the HeLP method was not sufficient \cite{Mar19}. 

Since there is a reduction of (PQ) to almost simple groups, one obvious way is to use the Classification of Finite Simple Groups to prove (PQ). One important class to deal with are those groups that have an alternating group as a socle. While considering those groups an interesting relation between the two methods described in this article turned up. Some generic results on units of sufficiently large order were achieved using the HeLP method \cite{BC17} and solely using HeLP, it was proved that (PQ) holds for all $A_n$ and $S_n$ for $n \leqslant 17$ (with the exception of $A_6$ handled above) \cite{Sal11, Sal13, BC17}. Thus there was very strong computational evidence that HeLP should be able to answer (PQ) affirmatively for all $A_n$ and $S_n$. But the combinatorial nature of character formulas for symmetric groups makes it hard to apply them in arithmetical formulas such as in Theorem~\ref{th:LP}. Further developing ideas from \cite{4primaryII} and \cite{Mar19} the authors eventually showed the following:
{\begin{theorem} \cite{BM19}
(PQ) holds for all almost simple groups with an alternating socle. In particular, (PQ) holds for all alternating and all symmetric groups.           
 \end{theorem}

Recently, the generic result of \cite{BM19} was generalized to the following clear result that answers (PQ) locally at the vertex $p$ only depending on the Sylow $p$-subgroup of $G$.

\begin{theorem}\label{th:block1} \cite{CM19} Assume the Sylow $p$-subgroup of a finite group $G$ is of order $p$ for some prime $p$. Let $q$ be any prime. Then $\V(\ZZ G)$ contains an element of order $pq$ if and only if $G$ contains an element of order $pq$. 
\end{theorem}

This theorem can be used to directly prove (PQ) for all alternating and symmetric groups (with the exception of degree 6), 24 of the 26 sporadic simple groups (with two applications of HeLP needed) and several series' of almost simple groups of Lie type of small rank, see \cite{CM19}.    

Theorem~\ref{th:block1} follows from a quantitative observation which works more general for blocks of defect $1$ and which can also be applied for units which are not necessarily of order $pq$. Recall that $\mu(\zeta,u,\chi)$ denotes the multiplicity of $\zeta$ as an eigenvalue of $D(u)$ where $u \in \V(\ZZ G)$ and $D$ is an ordinary representation of $G$ with character $\chi$. The next theorem uses Brauer trees in its formulation and we refer the reader to \cite[Section 4.12]{LP10} for a description of the theory for blocks of defect 1 which is the case relevant for us. 

\begin{theorem}\label{th:main_inequality}
Let $p$ be an odd prime and $B$ a $p$-block of defect 1 of a finite group $G$ with non-exceptional ordinary characters $\chi_1, \ldots ,\chi_e$  such that $\chi_1$ is a leaf of the Brauer tree of $B$. Let $\theta_1, \ldots ,\theta_t$ be the exceptional characters in $B$. Set $\chi_{e+1}=\theta_1+ \ldots +\theta_t$. Let $\delta_i$ be the sign of $\chi_i$ such that $\delta_{1}=1$. Let $u\in \V(\ZZ G)$ be of order $pm$ with $p$ not dividing $m$, let $\xi$ be any $m$-th root of unity and $\zeta$ a primitive $p$-th root of unity. Then 
$$0 \leqslant \mu(\xi , u, \chi_1) + \delta_{e+1}\cdot \mu(\xi \cdot \zeta, u, \chi_{e+1})+ t \sum_{i=1}^{e}\delta_i\cdot \mu(\xi \cdot \zeta, u, \chi_i).$$ 
\end{theorem}

An example of the application of this theorem is given in the next subsection and also in Section~\ref{sec:A7}.

\subsection{More open problems and an example}
As evident from the discussion above, the lattice method has not been applied by many researchers yet, in particular for questions different from (PQ). Hence one can hope that studying further examples will lead to new ideas. A generalization of Theorem~\ref{th:block1} for other types of Sylow $p$-subgroups would be of interest, this might need an improved understanding on how the eigenvalues of a unit influence the structure of modules it acts on, i.e.\ (2) in the process described above.

\begin{problem}\label{prob:more_gen_syl} Can Theorem~\ref{th:block1} be generalized for other Sylow $p$-subgroups? In particular for cyclic groups or elementary-abelian groups of rank $2$? \end{problem}

In fact an equivalent formulation of Theorem~\ref{th:block1} more accurately reflecting its proof, would be to assume that the principal $p$-block of $G$ has defect group isomorphic to $C_p$. Thus a solution of Problem~\ref{prob:more_gen_syl} would probably involve some deeper knowledge of modular representation theory, especially for non-cyclic Sylow $p$-subgroups.

(PQ) also remains open for some simple groups of rather small order.

\begin{problem} Does (PQ) have a positive answer for $\operatorname{PSL}(2,27)$? Does it have a positive answer for $\operatorname{Sz}(32)$, i.e.\ the Suzuki group defined over a field with 32 elements?  \end{problem}

When using the reduction of (PQ) to almost simple groups, one also has to take care of those groups, where the socle is a sporadic simple groups. As mentioned above, major progress was made using the lattice method and, in particular, Theorem~\ref{th:block1}. However, two groups withstood these attempts, namely the sporadic simple O'Nan group $O'N$ and the sporadic simple monster group $M$. Note that both have trivial outer automorphism group. Hence an obvious question is the following:
\begin{problem} Does (PQ) have an affirmative answer for 
 the sporadic simple O'Nan group $O'N$ and the sporadic simple Monster group $M$.
\end{problem}

Of course for almost simple groups of Lie type many special techniques are available which have rarely been used in the study of the unit groups of their integral group rings. As there are only two almost simple groups left which are not of Lie type for which (PQ) remains open, this raises the following question.

\begin{problem} How can the generic knowledge about a finite group of Lie type $G$ be used to obtain results on the structure of $\V(\ZZ G)$?\end{problem}

In combination with the HeLP method one might think e.g. of using generic characters such as given by Harish-Chandra or Deligne-Lusztig induction, cf. \cite{DM91} for an introduction to the topic. 

A further possibility would be a more systematic application of the lattice method not only to study (PQ), but also other problems, such as (ZC1), (KP) or the Spectrum Problem. We give an example of such an application.

\begin{proposition}
Let $G=J_1$ be the first sporadic simple Janko group. Then the Spectrum Problem has a positive answer for $G$, i.e.\ the finite orders of elements in $G$ and $\V(\ZZ G)$ coincide.
\end{proposition}

\begin{proof}
The HeLP method was used to investigate (ZC1) for $G$ in \cite{BJK11} and these results can quickly be reproduced using the command \texttt{HeLP\_ZC} for \texttt{C := CharacterTable("J1")}. We obtain that to prove the result it remains to exclude the existence of elements of order $30$ in $\V(\ZZ G)$. We will use the notation for the conjugacy classes of $G$ as in the above character table. Looking at the output of \texttt{HeLP\_sol[30];} we see that it is in fact sufficient to prove that there is no element $u \in \V(\ZZ G)$ of order $6$ such that $(\varepsilon_{2a}(u), \varepsilon_{3a}(u), \varepsilon_{6a}(u)) = (-2,3,0)$ as the $5$-th power of any element of order 30 in $\V(\ZZ G)$ would be conjugate in $\QQ G$ to $u$.

The Sylow $3$-subgroup of $G$ has order $3$ which means that the principal $3$-block $B$ of $G$ has defect $1$. The Brauer tree of $B$ has no exceptional vertex and the form
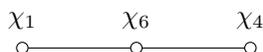
\[
\begin{tikzpicture}
\node[label=north:{$\chi_1$}] at (0,1.5) (1){};
\node[label=north:{$\chi_6$}] at (1.5,1.5) (2){};
\node[label=north:{$\chi_4$}] at (3,1.5) (3){};
\foreach \p in {1,2,3}{
\draw[fill=white] (\p) circle (.075cm);
}
\draw (.075,1.5)--(1.425,1.5);
\draw (1.575,1.5)--(2.925,1.5);
\end{tikzpicture}
\] 
where the ordinary characters labeling the tree have names as in the \textsf{GAP} character table library. This Brauer tree can be obtained from the decomposition matrix of $G$ with respect to the prime $3$, so via \texttt{DecompositionMatrix(C mod 3);} in \textsf{GAP}. Alternatively one can consult \cite{HL89}. We are hence in a position to apply Theorem~\ref{th:main_inequality}. Here $\delta_1 = \delta_4 = 1$, $\delta_6 = -1$ and there is no exceptional vertex, i.e. $t = 1$ and one can choose $\chi_{e+1}$ among the characters $\{\chi_1, \chi_4, \chi_6 \}$.

To do that using the \textsf{HeLP}-package set \texttt{s := [[1],[1],[2,-3,0]];} which is just the distribution of partial augmentations corresponding to $u$. Denote by $D_i$ a representation of $G$ corresponding to the character $\chi_i$ and let $\zeta$ be a primitive complex $3$-rd root of unity. Then we can run the command \texttt{HeLP\_MultiplicitiesOfEigenvalues(Irr(C)[i], 6, s);} for $i \in \{4,6\}$ to obtain the multiplicities $\mu(\zeta, u, \chi_6)$ and $\mu(\zeta, u, \chi_4)$ as $16$ and $14$ respectively. Clearly $D_1(u) = 1$ and thus $\mu(\zeta, u, \chi_1) = 0$, $\mu(1, u, \chi_1) = 1$. So setting $\xi = 1$ in Theorem~\ref{th:main_inequality} we get the inequality
$$0 \leqslant \mu(1, u, \chi_1) + \mu(\zeta, u, \chi_1) - \mu(\zeta, u, \chi_6) + \mu(\zeta, u, \chi_4) = -1,$$
which is a contradiction. 
\end{proof}

\section{Alternating Group of Degree 7}\label{sec:A7}

We will show how much can be achieved in the study of (ZC1) for the alternating group of degree $7$ using the HeLP method and the lattice method. Both methods turn out to be able to handle certain cases, but one critical case remains unsolved. By exhibiting this case we hope to inspire the search for an additional argument. We will not carry out all calculations needed for the HeLP method in detail, but only calculate some examples and give the \textsf{GAP}-commands which can be used to achieve the rest. 

We set $G= A_7$ throughout this paragraph. We will use \textsf{GAP}-notation for the conjugacy classes of $G$. We prove:

\begin{proposition}\label{prop:A7}
Let $G= A_7$ and assume that $u \in \V(\ZZ G)$ is an element of finite order which is not conjugate to a trivial unit in $\QQ G$. Then $u$ is of order $4$ and the non-vanishing partial augmentations of $u$ are $\varepsilon_{2a}(u) = 2$ and $\varepsilon_{4a}(u) = -1$. Moreover $u^2$ is not conjugate in the $2$-adic group ring $\ZZ_2 G$ to a trivial unit.  
\end{proposition}

\begin{remark}
\begin{enumerate}
\item If one could prove that units of order $4$ in $\V(\ZZ G)$ are in fact conjugate in $\QQ G$ to trivial units, this would also imply the following theorem, cf.\ \cite[Remark on p. 298]{Mar17}: For any finite group $H$ such that the Sylow $2$-subgroup of $H$ has order at most $8$, any $2$-subgroup of $\V(\ZZ H)$ is conjugate in $\QQ H$ to a subgroup of trivial units.
\item The HeLP method has previously been used to study (ZC1) for $G$ in \cite{Sal11}, but Salim did not show that $\V(\ZZ G)$ contains no elements of order $12$ and that elements of order $6$ in $\V(\ZZ G)$ are conjugate in $\QQ G$ to trivial units. This will explicitly be done below.
\end{enumerate}
\end{remark}

The exponent of $G$ equals $4 \cdot 3 \cdot 5 \cdot 7$ and the orders of non-trivial elements in $G$ are $\{2,3,4,5,6,7\}$. The non-trivial conjugacy classes in $G$ are $2a$, $3a$, $3b$, $4a$, $5a$, $6a$, $7a$ and $7b$. Moreover the square of an element in $6a$ lies in $3a$. Hence by Theorem~\ref{th:CL} to prove (ZC1) for $G$ we would need to show:
\begin{itemize}
\item If $u \in \V(\ZZ G)$ is of order $2$, $3$, $4$, $5$, $6$ or $7$, then $u$ is conjugate in $\QQ G$ to an element of $G$.
\item There are no elements of order 10, 12, 14, 15, 21 and 35 in $\V(\ZZ G)$.
\end{itemize}

We first approach this problem using Theorem~\ref{th:MRSW} and Proposition~\ref{prop:partaugs}. If $u \in \V(\ZZ G)$ is of order $2$, then $\varepsilon_C(u) \neq 0$ is only possible, if $C = 2a$ by Proposition~\ref{prop:partaugs}. Thus by Theorem~\ref{th:MRSW} the unit $u$ is conjugate in $\QQ G$ to a trivial unit. The same way we obtain that elements of order $5$ in $\V(\ZZ G)$ are conjugate in $\QQ G$ to trivial units. Furthermore, as the Sylow $5$-subgroup of $G$ is of order $5$ and the Sylow $7$-subgroup is of order $7$, it follows by Theorem~\ref{th:block1} that $\V(\ZZ G)$ contains no elements of order 10, 14, 15, 21, or 35. This could also be proved using the HeLP method.

We give a sample calculation for Theorem~\ref{th:LP} to show that elements of order $3$ in $\V(\ZZ G)$ are conjugate in $\QQ G$ to trivial units. So assume that $u \in \V(\ZZ G)$ is of order 3. Then $\varepsilon_C(u) \neq 0$ is only possible, by Proposition~\ref{prop:partaugs}, if $C \in \{3a, 3b \}$. Let $\chi$ be an ordinary or $p$-Brauer character, for $p \neq 3$, of $G$. Moreover let $\zeta$ be a primitive complex $3$-rd root of unity. Then, by Theorem~\ref{th:LP} and Hertweck's modular extension, for any integer $\ell$ we have that
\begin{align}\label{eq:A7Ord3}
\frac{1}{3} \Big[ \Tr_{\QQ / \QQ}(\chi(u^3)) + \Tr_{\QQ(\zeta)/\QQ}\left( [ \varepsilon_{3a}(u)\chi(3a) + \varepsilon_{3b}(u)\chi(3b) ] \zeta^{-\ell}\right)  \Big]
\end{align}
is a non-negative integer. The classes $3a$ and $3b$ are rational, i.e.\ elements of order $3$ in $G$ are conjugate to their squares, which means that $\chi(3a)$ and $\chi(3b)$ are integers. Hence $\Tr_{\QQ(\zeta)/\QQ}(\chi(3a)x) = \chi(3a)\Tr_{\QQ(\zeta)/\QQ}(x)$ for any element $x \in \QQ(\zeta)$. Furthermore, as $u$ is a normalized unit $1 = \sum_C \varepsilon_C(u) = \varepsilon_{3a}(u) + \varepsilon_{3b}(u)$ where the sum runs over all conjugacy classes $C$ in $G$. Hence substituting $\varepsilon_{3b}(u) = 1 - \varepsilon_{3a}(u)$, from \eqref{eq:A7Ord3} we get that
$$  \frac{1}{3} \Big[ \Tr_{\QQ / \QQ}(\chi(1)) + \left[\chi(3b) + \varepsilon_{3a}(u)(\chi(3a)-\chi(3b))\right]\Tr_{\QQ(\zeta)/\QQ}(\zeta^{-\ell}) \Big]$$
is a non-negative integer. Note that $\Tr_{\QQ(\zeta)/\QQ}(1) = 2$ and $\Tr_{\QQ(\zeta)/\QQ}(\zeta) = -1$. Hence putting $\ell = 0$ and $\ell = 1$ respectively we get that
$$ \frac{1}{3} \left( \chi(1) + 2 (\chi(3b) + \varepsilon_{3a}(u)(\chi(3a)-\chi(3b)) \right) $$
and
$$ \frac{1}{3} \left( \chi(1) - (\chi(3b) + \varepsilon_{3a}(u)(\chi(3a)-\chi(3b)) \right)$$
are non-negative integers. 

Now there is an irreducible $7$-modular Brauer character $\chi$ of $G$ such that $\chi(1) = 5$, $\chi(3a) = 2$ and $\chi(3b) = -1$. Hence we obtain that
$$\frac{1}{3} \left( 5 + 2 (-1 + 3\varepsilon_{3a}(u)) \right) = 1 + 2\varepsilon_{3a}(u) $$ 
and 
$$\frac{1}{3} \left( 5 - (-1 + 3 \varepsilon_{3a}(u)) \right) = 2 - \varepsilon_{3a}(u) $$
are non-negative integers, i.e. $-\frac{1}{2} \leqslant \varepsilon_{3a}(u) \leqslant 2$.

Moreover there is an irreducible $2$-modular character $\chi$ such that $\chi(1) = 4$, $\chi(3a) = -2$ and $\chi(3b) = 1$. This gives that
$$\frac{1}{3} \left( 4 + 2 (1 - 3\varepsilon_{3a}(u)) \right) = 2 - 2\varepsilon_{3a}(u) $$
and
$$ \frac{1}{3} \left( 4 - (1 -3 \varepsilon_{3a}(u)) \right) = 1 + \varepsilon_{3a}(u)$$
are non-negative integers, i.e.\ $-1 \leqslant \varepsilon_{3a}(u) \leqslant 1$. Taken together we conclude that $\varepsilon_{3a}(u) \in \{0,1 \}$ and $\varepsilon_{3b}(u) = 1 - \varepsilon_{3a}(u) \in \{1, 0\}$ which by Theorem~\ref{th:MRSW} means exactly that $u$ is conjugate in $\QQ G$ to a trivial unit.

We will not perform the analogues calculations for units of order $4$, $6$, $7$ and $12$. This can at once be done using the \textsf{HeLP}-package and the command \texttt{HeLP\_ZC(CharacterTable("A7"))}. This will give that elements of order 7 in $\V(\ZZ G)$ are conjugate in $\QQ G$ to trivial units and that elements of order 12 do not exist in $\V(\ZZ G)$. Moreover:
\begin{enumerate}
\item If $u \in \V(\ZZ G)$ is of order $4$ and not conjugate in $\QQ G$ to a trivial unit, then $\varepsilon_{2a}(u) = 2$ and $\varepsilon_{4a}(u) = -1$ and $u^2$ is conjugate to $2a$ in $\QQ G$.
\item If $u \in \V(\ZZ G)$ is of order $6$ and not conjugate in $\QQ G$ to a trivial unit, then: Either $u^2$ is conjugate in $\QQ G$ to an element in $3a$ and 
$$(\varepsilon_{2a}(u), \varepsilon_{3a}(u), \varepsilon_{3b}(u), \varepsilon_{6a}(u)) \in \{(-2,1,2,0), (2,0,0,-1) \}$$ or $u^2$ is conjugate in $\QQ G$ to an element in $3b$ and 
$$(\varepsilon_{2a}(u), \varepsilon_{3a}(u), \varepsilon_{3b}(u), \varepsilon_{6a}(u)) \in \{(-2,2,1,0), (0,1,-1,1), (2,-1,1,-1) \}.$$
\end{enumerate}

These are the results of the application of the HeLP method. Before we apply the lattice method we want to point out a difference with the results in \cite{Sal11}. Namely there the HeLP method was not sufficient to demonstrate that elements of order 12 do not exist in $\V(\ZZ G)$. The reason for this difference lies in the application of Proposition~\ref{prop:partaugs} (iii) which is used in the \textsf{HeLP}-package. We will give an example: If one would apply the HeLP method not using Proposition~\ref{prop:partaugs} (iii), then one could not exclude the existence of a unit $u \in \V(\ZZ G)$ of order 12 such that $u^2$ and $u^3$ are conjugate in $\QQ G$ to elements in $6a$ and $4a$, respectively, and $((\varepsilon_{2a}(u), \varepsilon_{3a}(u), \varepsilon_{3b}(u), \varepsilon_{4a}(u), \varepsilon_{6a}(u)) = (1,0,0,1,-1)$. But setting $p^j = 2$ and $D = 3a$ in Proposition~\ref{prop:partaugs} (iii) this would give
$$\varepsilon_{3a}(u) + \varepsilon_{6a}(u) \equiv \varepsilon_{3a}(u^2) \mod 2. $$
As $\varepsilon_{3a}(u^2) = 0$, this gives a contradiction.

We next apply the lattice method to show that if $u \in \V(\ZZ G)$ is of order $6$, then $u$ is conjugate in $\QQ G$ to a trivial unit. Assume $u \in \V(\ZZ G)$ is of order $6$ and denote by $\zeta$ a primitive complex $3$-rd root of unity. We will first use two characters from the principal $3$-block of $G$ in a similar way as it was done by Hertweck for $A_6$, cf. Section~\ref{sec:Lattice}.
We will use the parts of the character table of $G$ and the decomposition matrix with respect to the prime $3$ given in Table~\ref{tablePrincipalBlock3}. These are accessible in \textsf{GAP} using the commands \texttt{T := CharacterTable("A7");} and \texttt{DecompositionMatrix(T mod 3);} respectively. We will also use the numbering of characters as in the \textsf{GAP} character table library.

\begin{table}[h]
\caption{Part of the character table and decomposition matrix for $3$ of $A_7$}\label{tablePrincipalBlock3}\centering
\begin{subtable}[c]{0.3\textwidth} \centering
\subcaption{Part of the character table}
\begin{tabular}{l c c c c c} \hline
{ } & $1a$ & $2a$ & $3a$ & $3b$ & $6a$ \\ \hline \hline
$\chi_{1}$ & $1$ & $1$ & $1$ & $1$ & $1$  \\
$\chi_{5}$ & $14$ & $2$ & $2$ & $-1$ & $2$  \\
$\chi_{6}$ & $14$ & $2$ & $-1$ & $2$ & $-1$ \\ \hline
\end{tabular}
\end{subtable}
\qquad
\begin{subtable}[c]{0.42\textwidth} \centering
\subcaption{Part of the decomposition matrix for $3$}
\begin{tabular}{l c c} \hline
{ } & $\psi_{1}$  & $\psi_{5}$  \\ \hline\hline
$\chi_{1}$ & 1 & $\cdot$  \\
$\chi_{5}$ & 1 &  1 \\
$\chi_{6}$ & 1 & 1 \\ \hline
\end{tabular}
\end{subtable}
\end{table}

Denote by $D_i$ an ordinary representation of $G$ with character $\chi_i$. Note that for the characters in Table~\ref{tablePrincipalBlock3} we can assume that the representations are integral, as they extend to irreducible representations of $S_7$ which are all integral as Specht modules \cite[Theorem 4.12]{Jam86}. Hence whenever $\zeta$ has a certain multiplicity as an eigenvalue of a representation $D_i(u)$ or $D_i(u^2)$, then $\zeta^2$ must have the same multiplicity, as otherwise the character values would not be integral. Idem for $-\zeta$ and $-\zeta^2$.

We have $\chi_{i}(u^3) = 2$ for $i \in  \{5,6 \}$ and hence
$$D_{i}(u^3) \sim (1,1,1,1,1,1,1,1,-1,-1,-1,-1,-1,-1). $$
Assume first that $u^2$ is conjugate in $\QQ G$ to an element in $3a$. Then $\chi_{5}(u^2) = 2$ and $\chi_{6}(u^2) = -1$ giving
$$D_{5}(u^2) \sim (1,1,1,1,1,1,\zeta,\zeta^2,\zeta,\zeta^2,\zeta,\zeta^2,\zeta,\zeta^2) \ \ \text{and} \ \ D_{6}(u^2) \sim (1,1,1,1,\zeta,\zeta^2,\zeta,\zeta^2,\zeta,\zeta^2,\zeta,\zeta^2,\zeta,\zeta^2). $$
Note that that the eigenvalues of $D(u)$ are products of eigenvalues of $D(u^3)$ and $D(u^2)$.
Now we can argue as in the situation for $A_6$ described in Section~\ref{sec:Lattice} to see that the multiplicities of the eigenvalues $-1$, $-\zeta$ and $-\zeta^2$ must be the same for $D_{5}(u)$ and $D_{6}(u)$. Looking at the eigenvalues of $D_{5}(u^2)$ and $D_{6}(u^2)$ we see that this implies that the multiplicities of $\zeta$ and $\zeta^2$ as eigenvalue of $D_{6}(u)$ must be exactly $1$ bigger than for $D_{5}(u)$, e.g.\ if
$$D_{5}(u) \sim (1,1,1,1,1,1,\zeta,\zeta^2,-\zeta,-\zeta^2,-\zeta,-\zeta^2,-\zeta,-\zeta^2),$$
then
$$D_{6}(u) \sim (1,1,1,1,\zeta,\zeta^2,\zeta,\zeta^2,-\zeta,-\zeta^2,-\zeta,-\zeta^2,-\zeta,-\zeta^2).$$
This dependency of the eigenvalues of $D_{6}(u)$ on the eigenvalues of $D_{5}(u)$ can also be expressed by the equation 
\begin{align*}
3 &= \chi_{5}(u) - \chi_{6}(u) \\
 &= (2\varepsilon_{2a}(u) + 2\varepsilon_{3a}(u) - \varepsilon_{3b}(u) +2\varepsilon_{6a}(u)) - (2\varepsilon_{2a}(u) - \varepsilon_{3a}(u) + 2\varepsilon_{3b}(u) - \varepsilon_{6a}(u))  
\end{align*}
which gives
$$1 = \varepsilon_{3a}(u) - \varepsilon_{3b}(u) + \varepsilon_{6a}(u). $$
This equation is not satisfied by any of the tuples 
$$(\varepsilon_{2a}(u), \varepsilon_{3a}(u), \varepsilon_{3b}(u), \varepsilon_{6a}(u)) \in \{(-2,1,2,0), (2,0,0,-1) \}$$
which we computed using the HeLP method as the partial augmentations of $u$ in case $u^2$ is conjugate in $\QQ G$ to elements in $3a$, but $u$ is not conjugate to a trivial unit. We conclude that if $u^2$ is conjugate in $\QQ G$ to elements in $3a$, then $u$ is conjugate in $\QQ G$ to a trivial unit.

Next assume that $u^2$ is conjugate in $\QQ G$ to elements in $3b$. Then we compute  
$$D_{5}(u^2) \sim (1,1,1,1,\zeta,\zeta^2,\zeta,\zeta^2,\zeta,\zeta^2,\zeta,\zeta^2,\zeta,\zeta^2) \ \ \text{and} \ \ D_{6}(u^2) \sim (1,1,1,1,1,1,\zeta,\zeta^2,\zeta,\zeta^2,\zeta,\zeta^2,\zeta,\zeta^2) $$
(i.e. the eigenvalues from before have been interchanged). Arguing similarly as for the case before we obtain that the multiplicities of $-1$, $-\zeta$ and $-\zeta^2$ as eigenvalues of $D_{5}(u)$ and $D_{6}(u)$ must be the same. This gives then $-3 = \chi_{5}(u) - \chi_{6}(u)$ leading to
$$-1 = \varepsilon_{3a}(u) - \varepsilon_{3b}(u) + \varepsilon_{6a}(u). $$
This equation is not satisfied if 
$$(\varepsilon_{2a}(u), \varepsilon_{3a}(u), \varepsilon_{3b}(u), \varepsilon_{6a}(u)) \in \{(-2,2,1,0), (0,1,-1,1) \}.$$
From our results using HeLP we conclude that if $u$ exists then $(\varepsilon_{2a}(u), \varepsilon_{3a}(u), \varepsilon_{3b}(u), \varepsilon_{6a}(u)) =  (2,-1,1,-1)$.

We will handle this remaining case using Theorem~\ref{th:main_inequality}. The group $G$ possesses a $3$-block $B$ of defect $1$ containing the ordinary characters $\chi_2$, $\chi_7$ and $\chi_8$ which can be observed using the \textsf{GAP} commands \texttt{PrimeBlocks(T, 3).defect;} and \texttt{PrimeBlocks(T, 3).block;}. The Brauer tree of $B$ has the form
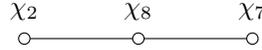
\[
\begin{tikzpicture}
\node[label=north:{$\chi_2$}] at (0,1.5) (1){};
\node[label=north:{$\chi_8$}] at (1.5,1.5) (2){};
\node[label=north:{$\chi_7$}] at (3,1.5) (3){};
\foreach \p in {1,2,3}{
\draw[fill=white] (\p) circle (.075cm);
}
\draw (.075,1.5)--(1.425,1.5);
\draw (1.575,1.5)--(2.925,1.5);
\end{tikzpicture}
\] 
Set \texttt{s := [[1],[0,1],[2,-1,1,-1]];} which is the distribution of partial augmentations of elements of order $6$ in $\V(\ZZ G)$ which we are studying. Then \texttt{HeLP\_MultiplicitiesOfEigenvalues(Irr(C)[i], 6, s);} for $i \in \{2,7,8\}$ calculates the multiplicities necessary for the application of Theorem~\ref{th:main_inequality}. We find $\mu(-1, u, \chi_2) = 0$, $\mu(-\zeta, u, \chi_2) = 1$, $\mu(-\zeta, u, \chi_7) = 2$ and $\mu(-\zeta, u, \chi_8) = 4$. So with $\xi = -1$, Theorem~\ref{th:main_inequality} gives
$$0 \leqslant \mu(-1, u, \chi_2)  + \mu(-\zeta, u, \chi_2) - \mu(-\zeta, u, \chi_7) + \mu(-\zeta, u, \chi_8) = -1, $$
contradicting the existence of $u$.

To prove Proposition~\ref{prop:A7} it remains to show that if $u \in \V(\ZZ G)$ is of order $4$ such that $(\varepsilon_{2a}(u), \varepsilon_{4a}(u)) = (2,-1)$, then $u^2$ is not conjugate in the $2$-adic group ring $\ZZ_2 G$ to a trivial unit. As $u^2$ is of order 2, if it is conjugate in $\ZZ_2 G$ to an element of $G$, then it is conjugate to an element in $2a$. Denote by $\chi_2$ the irreducible complex character of $G$ of degree $6$ which corresponds just to the natural permutation representation of $G$ from which a trivial submodule has been canceled. Denote by $D_2$ the standard $\ZZ$-representation affording $\chi_2$ and by $L_2$ the corresponding $\ZZ G$-module. We will use the bar-notation for reduction modulo the ideal $2\ZZ$, i.e.\ $\overline{L}_2$ is an $\mathbb{F}_2 G$-module. Now there is an element $g \in 2a$ (think of $(1,2)(3,4)$) such that 
$$D_2(g) = \begin{pmatrix} 
0 & 1 & 0 & 0 & 0 & 0 \\ 
1 & 0 & 0 & 0 & 0 & 0 \\ 
0 & 0 & 0 & 1 & 0 & 0 \\ 
0 & 0 & 1 & 0 & 0 & 0 \\ 
0 & 0 & 0 & 0 & 1 & 0 \\ 
0 & 0 & 0 & 0 & 0 & 1 \\ 
\end{pmatrix} $$
Thus as $\mathbb{F}_2 \langle g \rangle$-module $\overline{L}_2$ decomposes into the direct sum of two trivial modules and two indecomposable modules of dimension $2$. Recall from Section~\ref{sec:Lattice} that indecomposable $kC_{p^\ell}$-modules, for a prime $p$ and a field $k$ of characteristic $p$, are basically Jordan blocks and note that $\left(\begin{smallmatrix} 0 & 1 \\ 1 & 0 \end{smallmatrix}\right)$ has Jordan normal form $\left(\begin{smallmatrix} 1 & 1 \\ 0 & 1 \end{smallmatrix}\right)$  over $\mathbb{F}_2$.

Now let $u \in \V(\ZZ G)$ be of order $4$ such that $(\varepsilon_{2a}(u), \varepsilon_{4a}(u)) = (2,-1)$. Denote by $i$ a complex $4$-th root of unity. Then 
$$D_2(u) \sim \text{diag}(1,1,1,1,i,-i).$$
This can be computed by hand or using the \textsf{HeLP}-package via \texttt{s := [[1],[2,-1]]} (the distribution of partial augmentations of $u$) and \texttt{HeLP\_MultiplicitiesOfEigenvalues(Irr(C)[2], 4, s);} We next consider $L_2$ as a $\ZZ_2 \langle u \rangle$-module. The indecomposable $\ZZ_2 C_4$-modules are described in \cite[Section 34C]{CR81} or (more explicitly) in \cite[Section 4]{BG64}. It follows that the simple $\ZZ_2 \langle u \rangle$-modules are given by representations $u \mapsto 1$, $u \mapsto -1$ and $u \mapsto \left(\begin{smallmatrix} 0 & -1 \\ 1 & 0 \end{smallmatrix}\right)$. Note that under the last representation the image of $u$ has eigenvalues $i$ and $-i$. Moreover, an indecomposable $\ZZ_2 \langle u \rangle$-module possesses each simple module at most once as a composition factor \cite[Theorem 4.1]{BG64}. We conclude that $L_2$ as $\ZZ_2 \langle u \rangle$-module has at most one non-trivial indecomposable direct summand and this summand is of dimension $2$ (combining the eigenvalues $i$ and $-i$) or of dimension $3$ (combining $i$, $-i$ and $1$). Hence $\overline{L}_2$ as $k\langle \overline{u} \rangle$-module has also at most one non-trivial indecomposable direct summand and this summand is of dimension $2$ or $3$. 

Note that for dimension reasons an indecomposable $\mathbb{F}_2 \langle \overline{u} \rangle$-module of dimension $2$ or $3$ can not have more than one non-trivial direct indecomposable summand when viewed as $\mathbb{F}_2 \langle \overline{u}^2 \rangle$-module. This implies that $\overline{L}_2$ as $\mathbb{F}_2\langle \overline{u}^2 \rangle$-module has at most one non-trivial direct indecomposable summand. Hence the $\mathbb{F}_2 G$-module $\overline{L}_2$ has different isomorphism types as $\mathbb{F}_2 \langle \overline{u}^2 \rangle$ and $\mathbb{F}_2 \langle g \rangle$-module, meaning that $\overline{u}^2$ and $g$ are not conjugate in $\mathbb{F}_2 G$ and hence $u^2$ and $g$ are not conjugate in $\ZZ_2 G$.

We are not aware of an argument which could show that $u$ does not exist in $\ZZ G$.

\begin{problem} Prove that (ZC1) holds for $A_7$. \end{problem}

\bibliographystyle{amsalpha}
\bibliography{IndiaProc}

\end{document}